\documentclass{article}
\usepackage{amsmath,amssymb,amsthm,authblk,enumerate,graphicx,epstopdf}
\usepackage[a4paper, margin=2.5cm]{geometry}
\newtheorem{theorem}{Theorem}

\newtheorem{corollary}[theorem]{Corollary}
\newtheorem{proposition}[theorem]{Proposition}
\numberwithin{equation}{section}
\numberwithin{theorem}{section}
\begin{document}

\markboth{ISTV\'AN BAL\'AZS \& GERGELY R\"OST}{HOPF BIFURCATION AND PERIOD FUNCTIONS FOR WRIGHT-TYPE DELAY DIFFERENTIAL EQUATIONS}

\title{HOPF BIFURCATION AND PERIOD FUNCTIONS FOR WRIGHT-TYPE DELAY DIFFERENTIAL EQUATIONS }

\author{ISTV\'AN BAL\'AZS }

\affil{MTA-SZTE Analysis and Stochastics Research Group, Bolyai Institute, University of Szeged,\\
Aradi v\'ertan\'uk tere 1, Szeged, H-6720, Hungary\\
balazsi@math.u-szeged.hu}

\author{GERGELY R\"OST}
\affil{Mathematical Institute, University of Oxford,\\
Woodstock Road, OX2 6GG, Oxford, United Kingdom}
\affil{Bolyai Institute, University of Szeged,\\
	Aradi v\'ertan\'uk tere 1, Szeged, H-6720, Hungary\\
	rost@math.u-szeged.hu}
\maketitle

\begin{abstract}
We present the simplest criterion that determines the direction of the Hopf bifurcations of the delay differential equation $x'(t)=-\mu f(x(t-1))$, as the parameter $\mu$ passes through the critical values $\mu_k$. We give a complete classification of the possible bifurcation sequences. Using this information and the Cooke-transformation, we obtain local estimates and monotonicity properties of the periods of the bifurcating limit cycles along the Hopf-branches. Further, we show how our results relate to the often required property that the nonlinearity has negative Schwarzian derivative.
\end{abstract}


\section{Introduction}
\noindent The appearance of limit cycles around equilibria via Hopf bifurcations is a common phenomenon for delay differential equations, when a parameter of the equation is passing through a critical value and a pair of eigenvalues of the linearized system is crossing the imaginary axis on the complex plane. Depending on the nature of the nonlinearity, the Hopf bifurcation can be either supercritical or subcritical, i.e. the bifurcating periodic solution can be stable or unstable on the center manifold. It is well known how to determine the direction of the Hopf bifurcation for delay differential equations at least since the book of Hassard, Kazarinoff and Wan \cite{HKW}. One can use bilinear forms, center manifold reduction (see \cite{DvGVLW,HKW}), Lyapunov-Schmidt method \cite{WuGuo} or alternatively, the theory of normal forms for functional differential equations \cite{FM}. Based on these fundamental techniques, the literature of delay differential equations is vast by papers where Hopf bifurcation results are shown to many particular model systems arising from physical, engineering or biological applications. However, most of those articles provide only the complicated formula of the first Lyapunov coefficient, which is hard to relate to the original model parameters, and in fact if the reader wants to figure out the direction of the bifurcation in particular cases, it requires almost the same effort as repeating the whole calculation of the general formula. Also, due to the elaborative calculations, the literature of bifurcation theory is not free of minor mistakes or inaccuracies (some of those are discussed, for example, in \cite{LW} or \cite{Rost}).

To remedy this situation, our first aim here is to derive the simplest criterion for
the direction of the bifurcations for the class of scalar delay differential equations of the special form $$x'(t)=-\mu f(x(t-1)),$$ which will be then trivial to check in any specific situation.  Note that the equation $$z'(s)=-f(z(s-\mu))$$ can be easily rescaled into the previous one by the change of variables $s=t\mu$ and $x(t)=z(s)$, hence in the sequel we assume that the delay is one and $\mu$ will be our bifurcation parameter. This class of equations is frequently studied and includes such notorious examples as Wright's equation or the Ikeda-equation. We present concrete examples for all possible sequences of subcritical and supercritical Hopf bifurcations. Our calculations are based on the method of \cite{FM}.

Next, we use the formulae for the directions of the Hopf bifurcations and combine with the method of Cooke's transformation to obtain some information on the periods of the bifurcating solutions. In particular, we find narrow estimates of the period function along branches, and explore the relation between its monotonicity and the directions of the bifurcations. Finally, we explore the connection between our results and the Schwarzian derivative of the nonlinearity, which plays a significant role in many global stability results, and show that by local bifurcations one can not disprove the conjecture that local asymptotic stability implies global asymptotic stability for Wright-type delay differential equations with negative Schwarzian.

\section{Direction of Hopf bifurcation}
Consider the scalar delay differential equation
\begin{equation}x'(t)=-\mu f(x(t-1))=:g(x_t,\mu),\label{eq:1}\end{equation}
where $\mu \in \mathbb{R}$, $f$ is an $\mathbb{R}\to \mathbb{R}$, $C^3$-smooth function with $f(0)=0$ so it can be written as $$f(\xi)=\xi+B\xi^2+C\xi^3+h.o.t.,$$
where $B=f''(0)/2$ and $C=f'''(0)/6$. Note that $f'(0)=1$ can be assumed without the loss of generality, as we can normalize it via the parameter $\mu$. 
It is known that the direction of the Hopf bifurcation depends on the terms of the Taylor-series of the nonlinearity up to order three. In this case, the direction of the Hopf bifurcation around the zero equilibrium is determined by a relation between the coefficients $B$ and $C$. To our surprise, despite the method is well known for a much more general class of equations, we could not find a derivation of such a simple, readily available criterion for $B$ and $C$ in the literature for \eqref{eq:1}, only for the first Hopf bifurcation  in \cite{Stech} and in Chapter 6 of the recent book of H. L. Smith \cite{S}, and for a different class of equations in \cite{GZ}. The main result of this Section is the general condition for the stability of the Hopf bifurcation at any critical parameter value.
\begin{theorem}
\begin{enumerate}[a)]
\item Equation \eqref{eq:1} has Hopf bifurcations from the zero equilibrium at the critical parameter values $\mu_k=\frac{\pi}{2}+2k\pi,\ k\in\mathbb{Z}$.
\item The {\rm k}th bifurcation is
\begin{itemize}
\item supercritical if $C<\frac{22(4k+1)\pi-8}{15(4k+1)\pi}B^2$;
\item subcritical if $C>\frac{22(4k+1)\pi-8}{15(4k+1)\pi}B^2$.
\end{itemize}
\item If a Hopf bifurcation of Equation \eqref{eq:1} is
\begin{itemize}
\item supercritical, then the bifurcation branch starts to right if $\mu_k>0$ and left if $\mu_k<0$;
\item subcritical, then the bifurcation branch starts to left if $\mu_k>0$ and right if $\mu_k<0$.
\end{itemize}
\end{enumerate}
\label{th:1}
\end{theorem}

\begin{proof}
\begin{enumerate}[a)]
\item	The linearization  of equation \eqref{eq:1} is
	\begin{equation}x'(t)=-\mu x(t-1).\label{eq:lin}\end{equation}
	Searching for its solutions in form $x(t)=e^{\lambda t}$, we get
	$$\lambda e^{\lambda t}=-\mu e^{\lambda(t-1)},$$
	and the characteristic equation is
	$$\lambda=-\mu e^{-\lambda}.$$
	We would like to find the Hopf bifurcation points, so we substitute $\lambda=i\omega$, $\omega\in\mathbb R\setminus\{0\}$ and write
	$$i\omega=-\mu e^{-i\omega}=-\mu\cos\omega+i\mu\sin\omega.$$
	Taking real and imaginary parts, we get the following system of real equations
	\begin{eqnarray*}
		0&=&\mu\cos\omega,\\
		\omega&=&\mu\sin\omega.
	\end{eqnarray*}
	From the first of these equations, we get $\omega=\frac{\pi}{2}+n\pi$, $n\in\mathbb Z$. Substituting this into the second equation we have
	$$\frac{\pi}{2}+n\pi=\mu\sin\left(\frac{\pi}{2}+n\pi\right).$$
We distinguish two cases:
	\begin{itemize}
		\item$n=2k$, $k\in\mathbb Z$, in which case we find
		$$\frac{\pi}{2}+2k\pi=\mu\sin\left(\frac{\pi}{2}+2k\pi\right)=\mu;$$
		\item$n=2l+1$, $l\in\mathbb Z$, in which case we find
		\begin{eqnarray*}
			\frac{\pi}{2}+(2l+1)\pi=\mu\sin\left(\frac{\pi}{2}+(2l+1)\pi\right)=-\mu,
		\end{eqnarray*}
		which is equivalent to
		$$\frac{\pi}{2}-(2l+2)\pi=\mu .$$
			\end{itemize}
	We find that the via $k=-(l+1)$ the two cases can be treated together, and for the critical values we may just write $\mu_k=\frac{\pi}{2}+2k\pi=\frac{4k+1}{2}\pi$, $k\in\mathbb Z$.
For each $\mu_k$ there is a pair of critical eigenvalues $\pm i\omega_k$, where $\omega_k=\frac{4k+1}{2}\pi$. 

\item We follow the procedure developed in \cite{FM}. Let $L$ and $F$ be defined by the relation
$$L(\alpha)x_t+F(x_t,\alpha)=-(\mu_k+\alpha)f(x(t-1)),$$
where $x_t$ is the solution segment defined by $x_t(\theta)=x(t+\theta)$ for $\theta \in [-1,0]$, $L(\alpha)$ is a linear operator from $C([-1,0],\mathbb{R})$ to $\mathbb{R}$, $F$ is an operator from  $C([-1,0],\mathbb{R})\times \mathbb{R}$ to $\mathbb{R}$ with  $F(0,0)=0$ and $D_1F(0,0)=0$. We write
$$L(\alpha)=L_0+\alpha L_1+\frac{\alpha}{2}L_2+O(\alpha^3),$$
and for $(x_1,x_2,x_3,x_4) \in \mathbb{R}^4$,
\begin{eqnarray*} &F&(x_1e^{i\omega_k\theta}+x_2e^{-i\omega_k\theta}+x_3\cdot1+x_4e^{2i\omega_k\theta},0)
\\&&=B_{(2,0,0,0)}x_1^2+B_{(1,1,0,0)}x_1x_2+B_{(1,0,1,0)}x_1x_3+B_{(0,1,0,1)}x_2x_4
+B_{(2,1,0,0)}x_1^2x_2+\dots\end{eqnarray*}

Since $\mu_k=\omega_k=\frac{4k+1}{2}\pi$,  $e^{i\omega_k}=i$ holds.  For $\phi \in C([-1,0],\mathbb{R})$ we have
$$ L_0(\phi)=-\mu_k\phi(-1)=-\frac{4k+1}{2}\pi\phi(-1).$$
Hence,
\begin{eqnarray*}  &&L_0(1)=-\frac{4k+1}{2}\pi,\\&& L_0(\theta e^{i\omega_k\theta})=-\frac{4k+1}{2}\pi\left(-e^{-i\frac{4k+1}{2}\pi}\right)=-i\frac{4k+1}{2}\pi,
\\&&L_0(e^{2i\omega_k\theta})=-\frac{4k+1}{2}\pi\left(e^{-2i\frac{4k+1}{2}\pi}\right)
=\frac{4k+1}{2}\pi,\end{eqnarray*}
and the expansion of $F$ can be written as
\begin{eqnarray*}&&F(x_1e^{i\omega_k\theta}+x_2e^{-i\omega_k\theta}+
x_3\cdot1+x_4e^{2i\omega_k\theta},0)\\&&=-\frac{4k+1}{2}\pi\left(B(x_1(-i)+x_2i+x_3-x_4)^2+C(x_1(-i)+x_2i+x_3-x_4)^3+h.o.t.\right).\end{eqnarray*}
Then the $B_{(a,b,c,d)}$ coefficients are
\begin{eqnarray*} &&B_{(2,0,0,0)}=\frac{4k+1}{2}\pi B, \quad
 B_{(1,1,0,0)}=-
(4k+1)\pi B, \quad  B_{(1,0,1,0)}=(4k+1)\pi Bi, \\ &&B_{(0,1,0,1)}=(4k+1)\pi Bi,\quad
 B_{(2,1,0,0)}=\frac{3}{2}(4k+1)\pi Ci.\end{eqnarray*}
According to \cite{FM} (see formula (3.18) and Theorem 3.20), the direction of the bifurcation is determined by the sign of  
\begin{eqnarray*}  K=\mathrm{Re}\left[\frac{1}{1-L_0(\theta e^{i\omega_k\theta})}\left(B_{(2,1,0,0)}-\frac{B_{(1,1,0,0)}B_{(1,0,1,0)}}{L_0(1)}
+\frac{B_{(2,0,0,0)}B_{(0,1,0,1)}}{2i\omega_k-L_0(e^{2i\omega_k})}\right)\right]
.\end{eqnarray*}
We shall use the notation $a\sim b$ whenever $a=q b$ for some $q>0$. Substituting all terms into $K$, we need to find the sign of the real part of

\begin{eqnarray*} &&\frac{1}{1+i\frac{4k+1}{2}\pi}\Bigg(\frac{3}{2}(4k+1)\pi Ci-\frac{-(4k+1)\pi B(4k+1)\pi Bi}{-\frac{4k+1}{2}\pi}+\frac{\frac{4k+1}{2}\pi B(4k+1)\pi Bi}{i(4k+1)\pi-\frac{4k+1}{2}\pi}\Bigg)\\
&& = \frac{(2-(4k+1)\pi i)}{2(1+\frac{(4k+1)^2}{2^2})}(4k+1)\pi\left(\frac{3}{2}Ci-2B^2i+\frac{(-2i-1)B^2i}{5}\right)
\\&&\sim (2-(4k+1)\pi i)(4k+1)\left(\frac{3}{2}Ci+\frac{2}{5}B^2-\frac{11}{5}B^2i\right),\end{eqnarray*}
the latter expression having the real part
\begin{eqnarray*}(4k+1)\left[
(4k+1)\pi\frac{3}{2}C+\frac{4}{5}B^2-\frac{11}{5}(4k+1)\pi B^2\right]\sim\left[ C-\frac{22(4k+1)\pi-8}{15(4k+1)\pi}B^2\right].\end{eqnarray*}
\item To calculate in which direction a pair of characteristic roots crosses the imaginary axis at a bifurcation point, we differentiate the real part with respect to the parameter. Let us consider a parameter dependent solution of characteristic equation
$$\lambda=-\mu e^{-\lambda},$$
written as $\lambda(\mu)=\alpha(\mu)+i\omega(\mu)$,  where $\alpha$ and $\omega$ are the real and imaginary parts. Then we have
$$\alpha+i\omega=-\mu e^{-\alpha-i\omega}=-\mu e^{-\alpha}(\cos\omega-i\sin\omega).$$
Separating real and imaginary parts, we get
\begin{eqnarray*}
\alpha&=&-\mu e^{-\alpha}\cos\omega,\\
\omega&=&\mu e^{-\alpha}\sin\omega.
\end{eqnarray*}
Differentiating these equations with respect to $\mu$, we find
\begin{eqnarray*}
\alpha'&=&-e^{-\alpha}\cos\omega-\mu e^{-\alpha}(-\alpha')\cos\omega-\mu e^{-\alpha}(-\sin\omega)\omega',\\
\omega'&=&e^{-\alpha}\sin\omega+\mu e^{-\alpha}(-\alpha')\sin\omega+\mu e^{-\alpha}\cos(\omega)\omega'.
\end{eqnarray*}
Assume that the root is critical, i.e. $\alpha(\mu_k)=0$, and $\omega(\mu_k)=\omega_k$. As we have seen in part a), in the critical case $\cos\omega_k=0$ and $\sin\omega_k=1$. Then, evaluating the derivatives at a $\mu_k$, we obtain
\begin{eqnarray*}
\alpha'&=&\mu_k\omega',\\
\omega'&=&1-\mu_k\alpha'.
\end{eqnarray*}
Now we substitute $\omega'$ into the first equation, and express $\alpha'$ as
\begin{eqnarray*}
\alpha'(\mu_k)&=&\frac{\mu_k}{1+\mu_k^2}.
\end{eqnarray*}
Hence, $\alpha'(\mu_k)$ and $\mu_k$ has the same sign. This means that at a Hopf bifurcation, a pair of characteristic roots crosses the imaginary axis from left to right if and only if $\mu_k>0$. Hence the branch of a supercritical Hopf bifurcation starts to the right if and only if $\mu_k>0$, and the subcritical case is the opposite.
\end{enumerate}
\end{proof}

\begin{figure}[h]\label{fig1}
	\begin{center}\includegraphics[width=12cm]{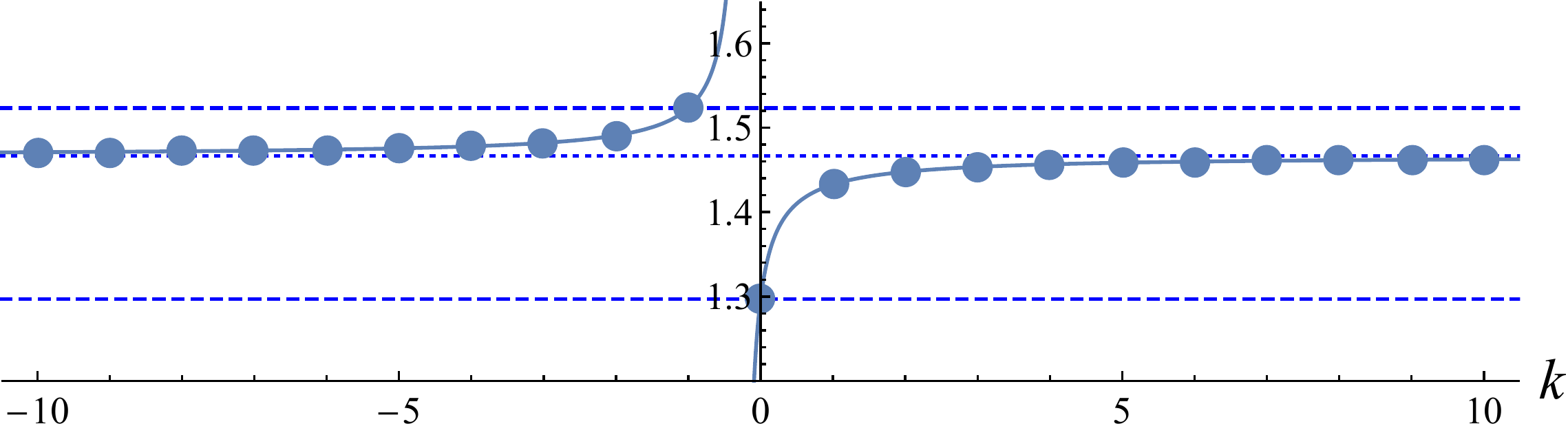}\end{center}
	\caption{Plot of  $H(k)=\frac{22(4k+1)\pi-8}{15(4k+1)\pi}$. The values are between $\frac{22\pi-8}{15\pi}$ and $\frac{66\pi+8}{45\pi}$, and they tend to $\frac{22}{15}$ as $k\to\pm\infty$. According to Theorem 1, when for some $k$, the value of $C/B^2$ is below $H(k)$, the $k$th bifurcation is supercritical.}
\end{figure}

\begin{figure}
\centering\small
\begin{tabular}{c c}
\includegraphics[height=1.5cm]{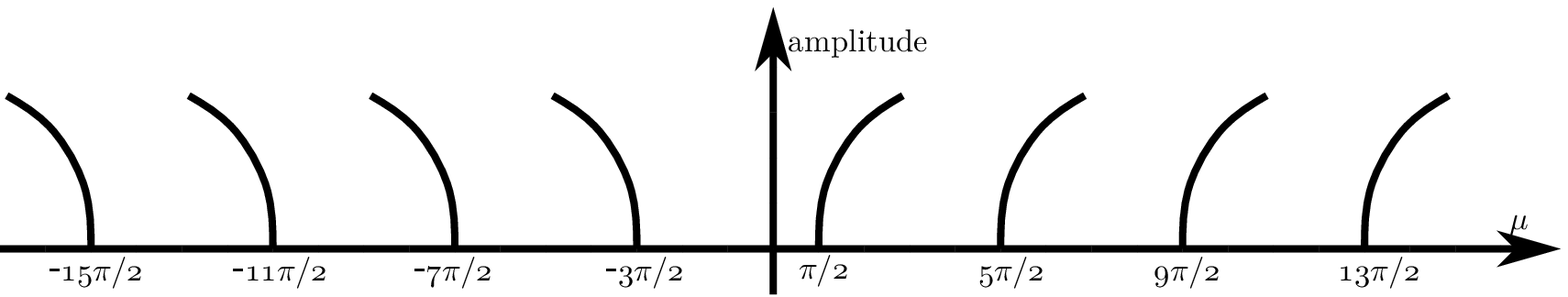}&\includegraphics[height=1.5cm]{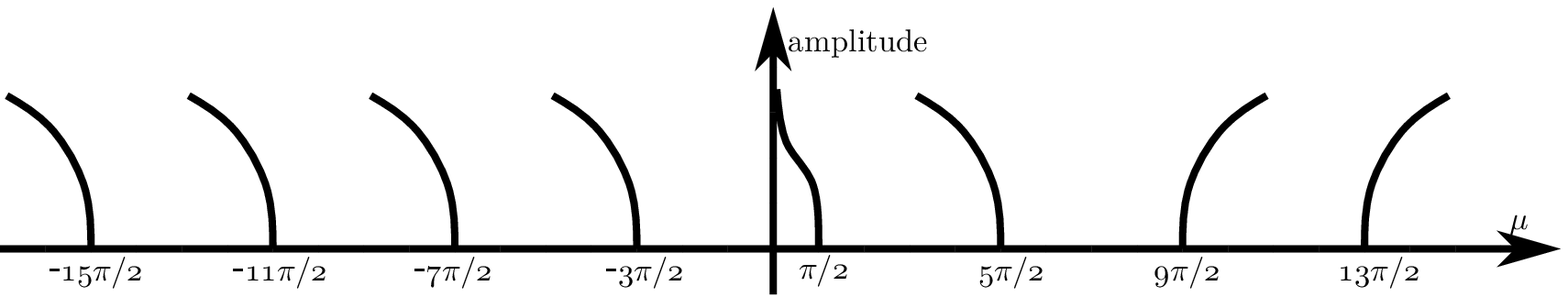}\\
(a) $C<\frac{22\pi-8}{15\pi}B^2$: All bifurcations are supercritical.&(b) $\frac{22\pi-8}{15\pi}B^2<C<\frac{22}{15}B^2$: There exists $n>0$ such that\\
&the $k$th bifurcation is subcritical iff $0\le k\le n$.\\\\
\includegraphics[height=1.5cm]{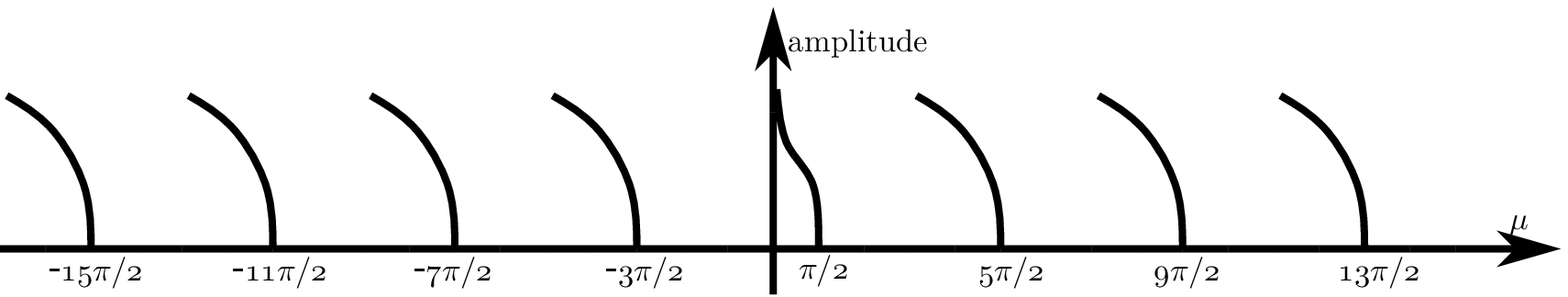}&\includegraphics[height=1.5cm]{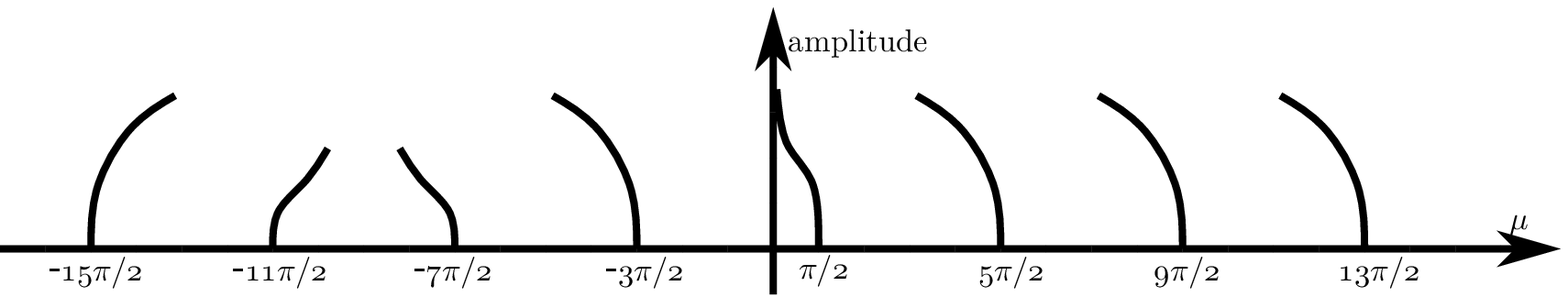}\\
(c) $C=\frac{22}{15}B^2$: The $k$th bifurcation is subcritical iff $k\ge0$.&(d) $\frac{22}{15}B^2<C<\frac{66\pi+8}{45\pi}B^2$: There exists $n<-1$ such that\\
&the $k$th bifurcation is subcritical iff $n\le k\le-1$.\\\\
\includegraphics[height=1.5cm]{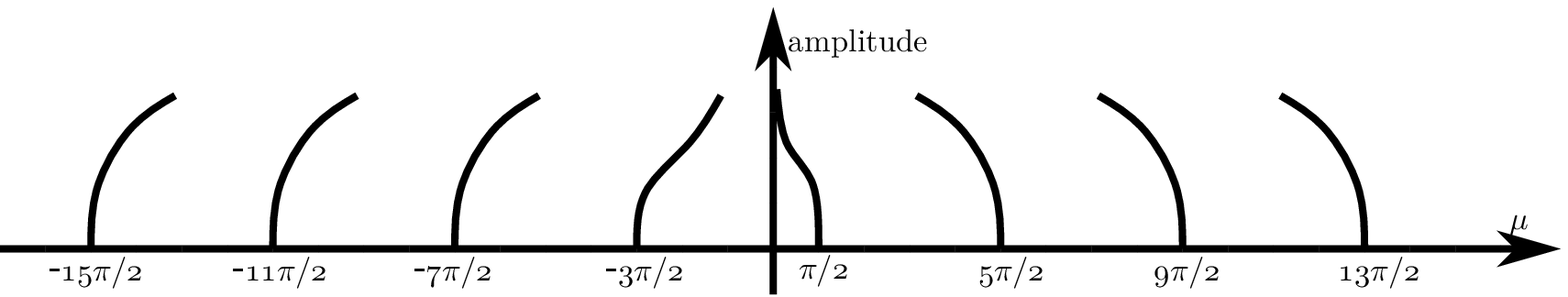}&{\includegraphics[height=1.5cm]{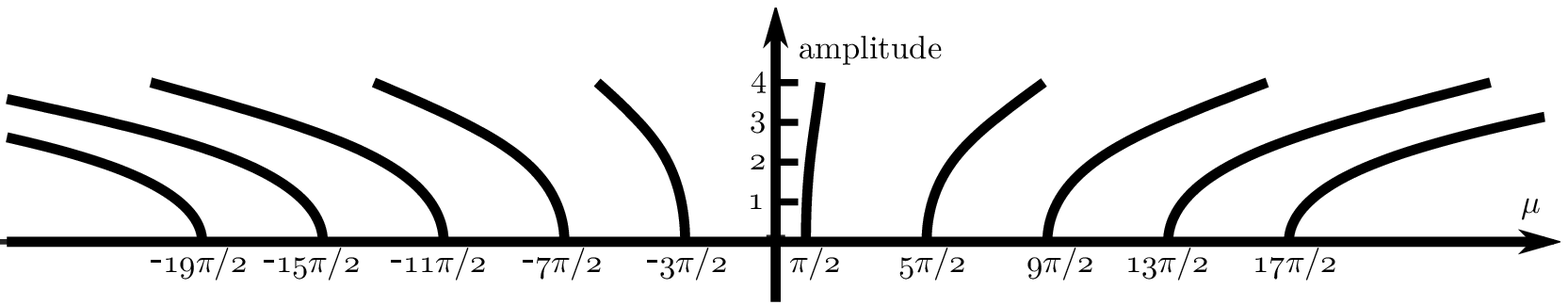}}\\
(e)  $C>\frac{66\pi+8}{45\pi}B^2$: All bifurcations are subcritical.&(f) Wright's equation

\end{tabular}
\caption{The possible configurations of bifurcation branches, based on Theorem 1.}
\end{figure}

As we mentioned, for the special case $k=0$, this result can be found in \cite{S}, page 97, which we now state as a corollary.

\begin{corollary}
The Hopf bifurcation at $\mu_0=\pi/2$ is supercritical if $C<\frac{22\pi-8}{15\pi}B^2$, and it is subcritical if  $C>\frac{22\pi-8}{15\pi}B^2$.
\label{cor:1}\end{corollary}

With the notation $H(k)=\frac{22(4k+1)\pi-8}{15(4k+1)\pi}$, we can see that the Hopf-bifurcation is supercritical when $C<H(k)B^2$. The function $H(k)$ is plotted in Figure 1, and from the shape of this function we easily find the following.

\begin{corollary}If $C<\frac{22\pi-8}{15\pi}B^2$ then every Hopf bifurcation is supercritical, if $C>\frac{66\pi+8}{45\pi}B^2$ then every Hopf bifurcation is subcritical.\label{cor:2}\end{corollary}
For convenience, note that $\frac{22\pi-8}{15\pi}\approx 1.3$ and $\frac{66\pi+8}{45\pi}\approx 1.52$. Theorem 1 and its corollaries allow us to give a complete classification of possible bifurcation sequences, which are depicted in Figure 2.

\section{Applications}
\subsection{Wright's equation}

The classical Wright-Hutchinson equation (also called delayed logistic equation) 
$$y'(t)=-\mu y(t-1)(1+y(t)), \quad \mu>0,$$
can be transformed into the form
$$x'(t)=-\mu(e^{x(t-1)}-1)$$
by the change of variable $x(t)=\ln(1+y(t))$, for solutions $y>-1$. This latter
equation is of type \eqref{eq:1} with
$f(\xi)=e^\xi-1$, $B=\frac{1}{2}$, $C=\frac{1}{6}$. Since
$C\approx0.167<\frac{22\pi-8}{15\pi}B^2\approx 0.324$, we can apply Corollary \ref{cor:2} to obtain the following fact (which was also derived in \cite{FM}, page 197).
\begin{corollary}
In Wright's equation, every Hopf bifurcation is supercritical.\label{cor:3}
\end{corollary}

\subsection{Ikeda equation}

The equation
$$y'(t)=-\sin(y(t-\mu))$$
arisen in the modeling of optical resonator systems. By rescaling, one has the equivalent form
$$x'(t)=-\mu\sin(x(t-1)),$$
which fits into \eqref{eq:1} with
$f(\xi)=\sin(\xi)$, $B=0$, $C=-\frac{1}{6}$. Since 
$C<0=\frac{22\pi-8}{15\pi}B^2=0$, Corollary \ref{cor:2} applies.
\begin{corollary}
In the Ikeda equation, every Hopf bifurcation is supercritical.\label{cor:4}
\end{corollary}
\begin{figure}[h]\label{fig8}
\centering
(a) \includegraphics[scale=0.2]{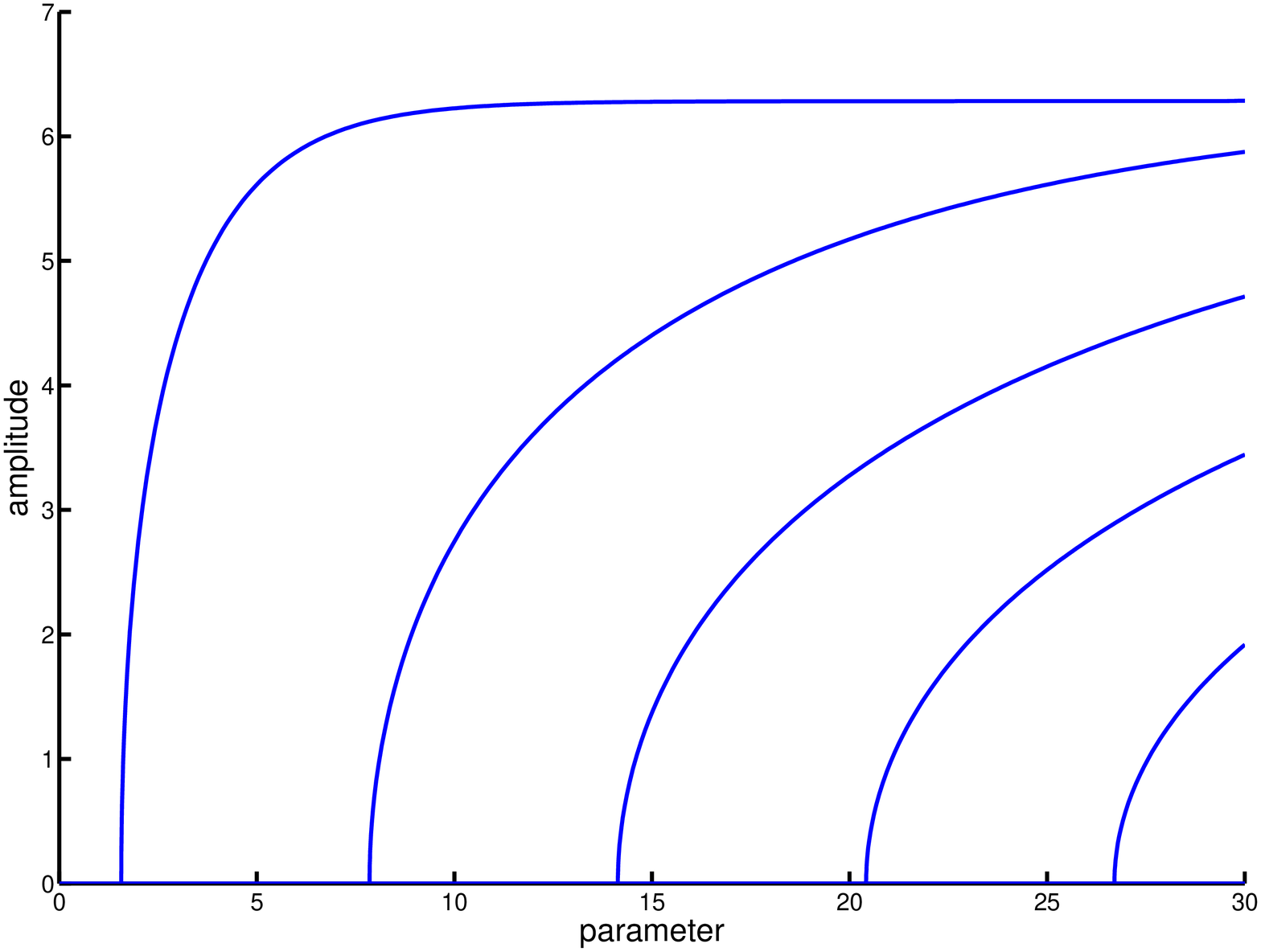} (b) \includegraphics[scale=0.2]{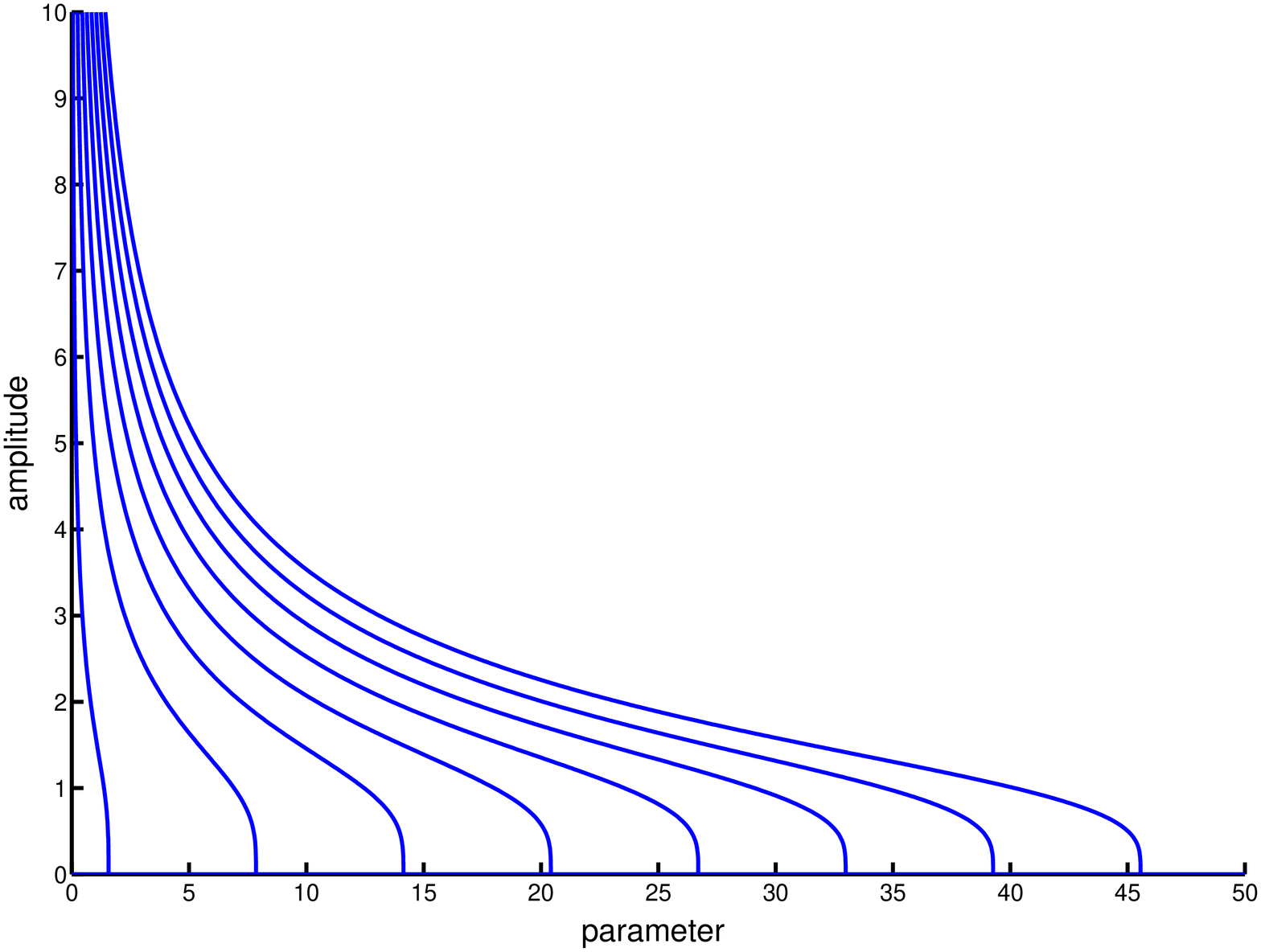}
\caption{(a) The bifurcation branches of the Ikeda equation. (b) The bifurcation branches of our totally subcritical example.}
\end{figure}

\subsection{A polynomial equation with criticality switching}\label{ex:switch}

Consider
$$x'(t)=-\mu(x(t-1)+x^2(t-1)+1.44x^3(t-1)),$$
that is $f(\xi)=\xi+\xi^2+1.44\xi^3$, $B=1$, $C=1.44$. Then
$$\frac{22(4\cdot1+1)\pi-8}{15(4\cdot1+1)\pi}<\frac{C}{B^2}<\frac{22(4\cdot2+1)\pi-8}{15(4\cdot2+1)\pi},$$
so the bifurcations at $\mu_0$ and $\mu_1$ are subcritical, the others are supercritical.

\subsection{A totally subcritical polynomial equation}

Consider
$$x'(t)=-\mu\left(x(t-1)+x^2(t-1)+\frac{22}{15}x^3(t-1)\right),$$
that is $f(\xi)=\xi+\xi^2+\frac{22}{15}\xi^2$, $B=1$, $C=\frac{22}{15}$. Then $\frac{C}{B}=\frac{22}{15}>\frac{22(4k+1)\pi-8}{15(4k+1)\pi}$, for all nonnegative integer $k$, so every Hopf bifurcation is subcritical for positive critical parameter values (and supercritical for negative parameter values).

\section{Period estimations}
Throughout this section we consider $\mu>0$. The following idea is known in the delay differential equation folklore as the Cooke-transform, which has been used for example in \cite{MPN,GK}. If $p(t)$ is a periodic solution of equation \eqref{eq:1} for parameter value $\mu=\mu_*>0$ with period $T$, then $q(t):=p((lT+1)t)$ is also a periodic solution of equation \eqref{eq:1} for parameter value $\mu=\mu_*(lT+1)$ with period $\frac{T}{lT+1}$, for any $l\in\mathbb N$.
This can be shown by the straightforward calculations
$$
q'(t)=-(\mu(lT+1))f(p((lT+1)t-lT-1))=-(\mu(lT+1))f(q(t-1))
$$
and
$$ q\left(t+\frac{T}{lT+1}\right)=p((lT+1)t+T)=p((lT+1)t)=q(t).$$
Thus we can define a map
$$C_l: (\mu_*,T,p(t)) \mapsto \left(\mu_*(lT+1),\frac{T}{lT+1},p((lT+1)t)\right),$$
where $p(t)$ is a periodic solution of equation \eqref{eq:1} with parameter value $\mu_*$ and period $T$.
\begin{proposition} Let $k,l \in \mathbb N$. Then near the bifurcation points, $C_l$ maps the kth bifurcation branch to the (k+l)th bifurcation branch. 
\end{proposition}
\begin{proof}
Consider the Hopf-branch of periodic solutions near the parameter value $\mu_k$, $k\in\mathbb N$. Let $\delta_k=1$ if the $k$th bifurcation is supercritical, and let $\delta_k=-1$ if the $k$th bifurcation is subcritical.
Then, for any $k$ there is a local branch of periodic solutions $p^k_\eta(t)$ corresponding to parameter value $\mu=\mu_k+\delta_k\eta$ with $\eta \in (0,\eta_k)$ with some $\eta_k>0$. The minimal period of $p^k_\eta$ is denoted by $T^k_\eta$. Recall from the previous section that at the critical values $\mu_k=\frac{\pi}{2}+2k\pi=\frac{4k+1}{2}\pi$, the critical eigenvalue is $i\omega_k=i\frac{4k+1}{2}\pi$, hence $T^k_\eta \to \frac{2\pi}{\omega_k}=\frac{4}{4k+1}=:T_0^k$ as $\eta \to 0$. Notice that 
$$\mu_k(lT_0^k+1)=\frac{4k+1}{2}\pi\left( l\frac{4}{4k+1}+1\right)=\frac{\pi}{2}(4(k+l)+1)=\mu_{k+l}$$
and
$$\frac{T_0^k}{lT_0^k+1}=\frac{1}{l+\frac{4k+1}{4}}=\frac{4}{4l+4k+1}=\frac{2\pi}{\omega_{k+l}}=T_0^{k+l}. $$

From the uniqueness of local branches (see \cite[Theorem X.2.7]{DvGVLW}), we find that the Cooke-transform maps Hopf bifurcation branches to Hopf bifurcation branches.  
\end{proof}

\begin{theorem}
In equation \eqref{eq:1}, if $k\ge0$ and the $k$th Hopf bifurcation is supercritical, then we have the following estimate on the period of the Hopf solution near $\mu_k$:
$$T^k_\eta\geq \frac{4}{4k+1+\frac{2\eta}{\pi}}.$$\label{th:super}
\end{theorem}
\begin{proof}
If the $k$th bifurcation is supercritical, then $\delta_k=1$, and by Corollary 2, all the $(k+l)$th bifurcations ($l \in \mathbb N$) are supercritical as well.
Then, taking into account Proposition 1, 
\begin{equation} (\mu_k+\eta)(lT^k_\eta+1)>\mu_{k+l},\label{ellentmondashoz}\end{equation}
that is
$$T^k_\eta>\left(\frac{\mu_{k+l}}{\mu_k+\eta}-1\right)l^{-1}=\frac{4-\frac{2\eta}{l\pi}}
{4k+1+\frac{2\eta}{\pi}}. $$
This inequality holds for any $l \in \mathbb N$, thus letting $l \to \infty$ we finish the proof. 
\end{proof}

\begin{theorem}
If in equation \eqref{eq:1} all Hopf bifurcations are subcritical and $k\ge0$, then we have the following estimate on the period of the Hopf solution near $\mu_k$:
$$T^k_\eta\leq \frac{4}{4k+1-\frac{2\eta}{\pi}}.$$\label{th:sub}
\end{theorem}
\begin{proof}
Now for any $k,l \in \mathbb{N}$, $\delta_k=\delta_l=-1$, and by Proposition 1, 
$$(\mu_k-\eta)(lT^k_\eta+1)<\mu_{k+l},$$
that is
$$T^k_\eta<\left(\frac{\mu_{k+l}}{\mu_k-\eta}-1\right)l^{-1}=\frac{4+\frac{2\eta}{l\pi}}
{4k+1-\frac{2\eta}{\pi}}. $$
This inequality holds for any $l \in \mathbb N$, thus letting $l \to \infty$ we finish the proof. 
\end{proof}

\begin{theorem}\label{th:switch}
If $\frac{22\pi-8}{15\pi}B^2\le C<\frac{22}{15}B^2$ and $k\ge0$, then define
$$n:=\max\left\{m\in\mathbb N_0:\ C>\frac{22(4m+1)\pi-8}{15(4m+1)\pi}B^2\right\}.$$
If $k<n$ then near $\mu_k$ we have the estimates
$$\frac{4+\frac{2\eta}{(n-k+1)\pi}}{4k+1-\frac{2\eta}{\pi}}<T_\eta^k<\frac{4+\frac{2\eta}{(n-k)\pi}}{4k+1-\frac{2\eta}{\pi}}.$$
If $k=n$ then we only have the lower estimate:
$$T_\eta^k>\frac{4+\frac{2\eta}{\pi}}{4k+1-\frac{2\eta}{\pi}}.$$
\end{theorem}
\begin{proof}
Assume that $k\le n$. Then the $k$th bifurcation is subcritical, $\delta_k=-1$.

First suppose that $l_1>0$ and the $(k+l_1)$th bifurcation is supercritical. Then
$$(\mu_k-\eta)(l_1T_\eta^k+1)>\mu_{k+l_1},$$
that implies
$$T_\eta^k>\frac{4+\frac{2\eta}{l_1\pi}}{4k+1-\frac{2\eta}{\pi}}.$$
Now we choose $l_1$ to be the minimal index which still gives a supercritical bifurcation, that is $l_1:=n-k+1$.
Next, suppose that the $(k+l_2)$th bifurcation is subcritical. This is only possible if $k<n$. Then
$$(\mu_k-\eta)(l_2T_\eta^k+1)<\mu_{k+l_2},$$
that implies
$$T_\eta^k<\frac{4+\frac{2\eta}{l_2\pi}}{4k+1-\frac{2\eta}{\pi}}.$$
Finally, choose $l_2$ to be the maximal index that still gives subcritical bifurcation, that is $l_2:=n-k$.
\end{proof}

In some situation this theorem provides very sharp estimations of the period function, which is illustrated in Figure 4. 

\begin{figure}\label{fig9}
\begin{center}\includegraphics[scale=0.4]{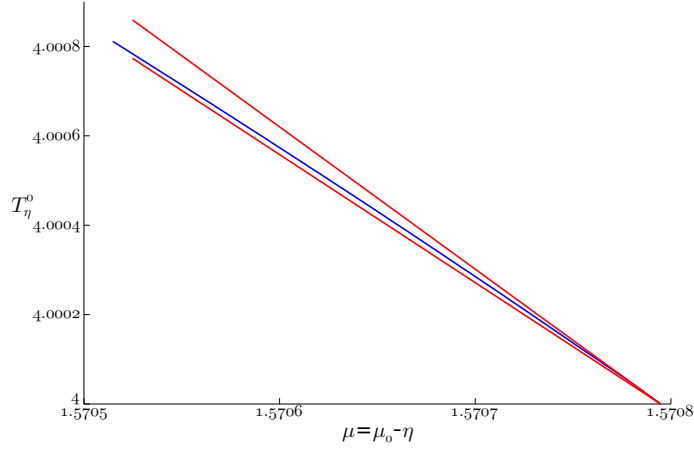}\end{center}
\caption{Narrow estimates on the period of the 0th branch of Example \ref{ex:switch} by Theorem \ref{th:switch} (red curves) compared to numerically obtained periods (blue curve).}
\end{figure}

\begin{corollary}\label{cor:subsub}
If in equation \eqref{eq:1} the $k$th Hopf bifurcation is subcritical for some $k\geq 0$, and the  periods satisfy $T^k_\eta<T^k_0$ near $\mu_k$, then for all $l>0$ the $(k+l)$th Hopf bifurcation is also subcritical.
\end{corollary}
\begin{proof}
If the $k$th Hopf bifurcation is subcritical, then 
$$(\mu_k-\eta)(lT^k_\eta+1)< \mu_k(lT^k_0+1)=\mu_{k+l}.$$
This means that the Cooke-transform maps the $k$th branch to the left side of $\mu_{k+l}$, thus the $(k+l)$th bifurcation is also subcritical.
\end{proof}

In the situation of $\frac{22\pi-8}{15\pi}B^2<C<\frac{22}{15}B^2$ (see Figure 2.b.), we can infer the monotonicity of the period functions at the subcritical bifurcations, as the next corollary shows.

\begin{corollary}\label{cor:subsub}
	If in equation \eqref{eq:1} the $k$th Hopf bifurcation is subcritical for some $k\geq 0$, but the $(k+l)$th Hopf bifurcation is supercritical for any $l>0$, then $T^k_\eta>0$ is monotone increasing for small $\eta$.
\end{corollary}
\begin{proof}
If the $k$th Hopf bifurcation is subcritical, but the $(k+l)$th is supercritical, then for $\eta_1<\eta_2$ we have 
	$$(\mu_k-\eta_1)(lT^k_{\eta_1}+1)<(\mu_k-\eta_2)(lT^k_{\eta_2}+1).$$
	This is possible only if $T^k_{\eta_1}<T^k_{\eta_2}$.
\end{proof}

\section{Schwarzian-derivative and the direction of the Hopf bifurcation}
The Schwarzian derivative of a $C^3$ function $f$ is defined as
$$(Sf)(\xi)=\frac{f'''(\xi)}{f'(\xi)}-\frac{3}{2}\left(\frac{f''(\xi)}{f'(\xi)}\right)^2$$
at points $\xi$ where $f'(\xi)\neq 0$. This quantity plays an important role in many results regarding the global dynamics of difference equations, which can be extended to delay differential equations in various cases (see \cite{LizRost,LizRost2,LizRost3,Liz} and references thereof). A global stability conjecture was formulated in \cite{Liz}, stating that the zero solution of \eqref{eq:1} is globally asymptotically stable whenever it is locally asymptotically stable, if $Sf<0$ and some other technical conditions hold (for related conjectures, see \cite{LizRost}). An obvious way to disprove this global stability conjecture would be the following:  find a nonlinearity $f$ with $Sf<0$, where the Hopf bifurcation of \eqref{eq:1} is subcritical at $\mu_0$. This would provide a counterexample. Since both the directions of the bifurcation and the sign of the Schwarzian are determined by the derivatives of the nonlinearity up to order three, in view of the results of the previous sections, it is most natural to make a comparison to check whether such a counterexample is possible.

\begin{corollary}
If $Sf<0$, then all Hopf bifurcations are supercritical. Furthermore, if $f''(0)=0$, then for any $k$, the $k$th bifurcation is supercritical if and only if $Sf(0)<0$. 
\end{corollary}
\begin{proof}
From the definition, it is easy to evaluate $Sf(0)=6(C-B^2)$, thus $Sf<0$ implies $Sf(0)<0$ and $C<B^2$. By Corollary 2, all Hopf bifurcations are supercritical. In the special case $f''(0)=0$, we have $B=0$ and $Sf(0)=6C$,
thus both the sign of the Schwarzian and the direction of the bifurcation are determined by the sign of $C$.  
\end{proof}
We found that it is not possible to construct a counterexample to the conjecture of Liz et al. by means of a subcritical Hopf bifurcation.

\section{Acknowledgements}

IB was supported by Hungarian Scientific Research Fund OTKA K109782 and EU-funded Hungarian grant EFOP-3.6.1-16-2016-00008. GR was supported by NKFIH FK124016 and Marie Skłodowska-Curie Grant No. 748193. The authors thank Maria Vittoria Barbarossa and Jan Sieber for helping with DDE-BifTool.

\end{document}